\documentclass[11pt]{article}

\usepackage[margin=25mm]{geometry}
\usepackage{amsmath,amssymb,amsthm,mathtools}
\usepackage{booktabs}
\usepackage{enumerate}
\usepackage{xcolor}
\usepackage{kbordermatrix}

\usepackage{microtype}
\usepackage{etoolbox}

\usepackage[hidelinks]{hyperref}
\usepackage{cleveref}

\def\wtd{\widetilde}
\def\what{\widehat}

\usepackage{amsbsy}
\def\ba{\pmb{a}}
\def\bb{\pmb{b}}

\def\bx{\pmb{x}}
\def\by{\pmb{y}}

\def\blambda{\pmb{\lambda}}

\def\bsigma{\pmb{\sigma}}

\DeclareMathOperator{\rmd}{d}

\def\sss{\scriptstyle}

\newtheorem{theorem}{Theorem}[section]
\newtheorem{lemma}{Lemma}[section]

\theoremstyle{definition}
\newtheorem{definition}{Definition}[section]

\newtheorem{example}{Example}[section]

\newcommand{\bbC}{\mathbb{C}}
\newcommand{\bbI}{\mathbb{I}}
\newcommand{\bbR}{\mathbb{R}}

\DeclareMathOperator{\diag}{diag}

\DeclareMathOperator{\eig}{eig}
\DeclareMathOperator{\gap}{gap}
\DeclareMathOperator{\rank}{rank}
\DeclareMathOperator{\sv}{sv}
\DeclareMathOperator{\tr}{tr}
\DeclareMathOperator{\UI}{ui}

\DeclareMathOperator{\F}{F}
\DeclareMathOperator{\T}{T}

\usepackage{mathrsfs}
\def\scrJ{\mathscr{J}}

\newcommand{\Rea}{\operatorname{Re}}

\numberwithin{equation}{section}

\title{A Sharp Unitarily Invariant Norm Bound for the Off-Diagonal\\
       Block Perturbation of a Hermitian Matrix}

\author{
Lei-Hong Zhang%
\thanks{\raggedright School of Mathematical Sciences, Soochow University, Suzhou 215006, Jiangsu, China. Email: {\tt longzlh@suda.edu.cn}.
             Supported in part by the National Natural Science Foundation of China
             NSFC-12471356 and NSFC-12371380.}
\and
Ren-Cang Li%
\thanks{\raggedright Department of Mathematics,
University of Texas at Arlington, Arlington, TX 76019-0408, USA.
{Email: \texttt{rcli@uta.edu}.}}
}

\date{{October 4, 2026}}

\begin{document}
\maketitle

\begin{abstract}
Let
$$
A=\begin{bmatrix} H_1 & E^* \\ E & H_2 \end{bmatrix}
\quad\text{and}\quad
\widetilde A=\begin{bmatrix} H_1 & 0 \\ 0 & H_2 \end{bmatrix}
$$
{be partitioned Hermitian matrices, with $\widetilde A$ obtained by dropping the off-diagonal blocks of $A$. Let $\eta$ be the minimum distance between the spectra of $H_1$ and $H_2$; these spectra may interlace. Define, for $\delta,\epsilon\ge0$,}
$$
\phi(\delta,\epsilon)=
\begin{cases}
  2\epsilon/(\delta+\sqrt{\delta^2+4\epsilon^2}), &\quad\mbox{if $(\delta,\epsilon)\ne (0,0)$}, \\
  1, &\quad\mbox{if $(\delta,\epsilon) = (0,0)$},
\end{cases}
$$
and let $\what E=A-\widetilde A$ and $\epsilon_2=\|E\|_2=\|\what E\|_2$, the matrix spectral norm.
Li and Li [{\em Linear Algebra Appl.}, 395:183--190, 2005] established a sharp spectral-norm bound on the changes in the eigenvalues of $A$:
$$
\big\|\diag\big(\pmb{\lambda}(A)-\pmb{\lambda}(\widetilde A)\big)\big\|_2
\le \phi(\eta,\epsilon_2)\,\|E\|_2,
$$
where $\pmb{\lambda}(A)$ is the vector whose components are the eigenvalues of $A$ in descending order
and similarly for $\pmb{\lambda}(\widetilde A)$. 
In this paper,
this bound is extended to every unitarily invariant norm $\|\cdot\|_{\UI}$:
$$
\big\|\diag\big(\pmb{\lambda}(A)-\pmb{\lambda}(\widetilde A)\big)\big\|_{\UI}
\le \phi(\eta,\epsilon_2)\,\|A-\widetilde A\|_{\UI}.
$$
A $2\times2$ example attains equality for every such norm.
An analogous bound for singular values is also established when the
off-diagonal blocks of a rectangular matrix are dropped.
\end{abstract}

\medskip
{\small
\medskip

}

\clearpage
\section{Introduction}
\label{sec:intro}
Consider  a partitioned Hermitian matrix
\begin{equation}\label{eq:A}
A=\begin{bmatrix}
   H_1 & E^* \\
   E & H_2
   \end{bmatrix}\in\bbC^{N\times N},
\end{equation}
where $H_1\in\bbC^{m\times m}$ and $H_2\in\bbC^{n\times n}$ are Hermitian,
{$E\in\bbC^{n\times m}$, $E^*$ is its conjugate transpose, $m,n\ge1$, and $N=m+n$.}
{A basic question in eigenvalue computation is how much the eigenvalues of $A$ change when the off-diagonal blocks are dropped, splitting the eigenvalue problem into the two independent problems for $H_1$ and $H_2$.}
There are a number of results in matrix perturbation theory that can be used to bound the changes in eigenvalues
\cite{bhat:1997,daka:1970,demm:1997,govl:2013,li:2014HLA,math:1998,parl:1998,stsu:1990}.
In particular, Li and Li~\cite{lili:2005} established a sharp bound on
the change in each individual eigenvalue of $A$. In this paper, our goal is to extend the main result in \cite{lili:2005}
from the spectral norm to unitarily invariant norms.

Dropping   $E$ and $E^*$ from $A$, we get
\begin{equation}
\label{eq:Atilde}
\wtd A
=
\begin{bmatrix}
H_1 & 0 \\
0 & H_2
\end{bmatrix}
\end{equation}
which is composed of the block-diagonal part of $A$; let
\begin{equation}
\label{eq:whatE}
\what E:=A-\wtd A=
\begin{bmatrix}
0 & E^* \\
E & 0
\end{bmatrix}.
\end{equation}
Denote the eigenvalues of $A$ and $\wtd A$, in decreasing order, by
$$
\lambda_1(A)\ge\cdots\ge\lambda_N(A),
\quad
\lambda_1(\wtd A)\ge\cdots\ge\lambda_N(\wtd A),
$$
respectively, and by $\eig(A)$ the multiset of eigenvalues of $A$, i.e.,
$\eig(A)=\{\lambda_i(A)\}_{i=1}^N$ and similarly for $\eig(\wtd A)$.
Pack these eigenvalues into two column vectors
$$
\blambda(A)=[\lambda_1(A),\ldots,\lambda_N(A)]^{\T},\quad
\blambda(\wtd A)=[\lambda_1(\wtd A),\ldots,\lambda_N(\wtd A)]^{\T}\in\bbR^N
$$
for $A$ and $\wtd A$, respectively,
where $(\cdot)^{\T}$ takes the matrix/vector transpose, and
define
\begin{subequations}\label{eq:gaps}
\begin{align}
\eta_i &:=
     \begin{cases}
     \min\limits_{\mu_2\in\eig(H_2)}|\lambda_i(\wtd A)-\mu_2|,
                    &\mbox{if $\lambda_i(\wtd A)\in\eig(H_1)$}, \\
     \min\limits_{\mu_1\in\eig(H_1)}|\lambda_i(\wtd A)-\mu_1|,
                    &\mbox{if $\lambda_i(\wtd A)\in\eig(H_2)$},
     \end{cases}  \label{eq:gaps-1} \\
\eta&:=\gap\bigl(\eig(H_1),\eig(H_2)\bigr)
   :=\min_{\mu_1\in\eig(H_1),\,\mu_2\in\eig(H_2)}|\mu_1-\mu_2|=\min_{1\le i\le N}\eta_i.
\label{eq:gaps-2}
\end{align}
\end{subequations}
Thus $\eta_i$ is the distance from $\lambda_i(\wtd A)$ to the spectrum
of the opposite diagonal block. {If this eigenvalue belongs to both block spectra, both branches in \eqref{eq:gaps-1} give $\eta_i=0$.}
The global gap $\eta$ is the minimum distance between the two block
spectra, which may interlace.
Throughout, $\|\cdot\|_2$  is the matrix spectral norm or the vector Euclidean
norm, and
$\|\cdot\|_{\F}$ is the matrix Frobenius norm.
Let
$$
\epsilon_2:=\|E\|_2=\|\what E\|_2, \quad
\epsilon_{\F}:=\|E\|_{\F}.
$$
Define, for $\delta,\epsilon\ge0$,
\begin{equation}\label{eq:phi}
\phi(\delta,\epsilon)
:=
\begin{cases}
\dfrac{2\epsilon}{\delta+\sqrt{\delta^2+4\epsilon^2}},
   &(\delta,\epsilon)\ne(0,0),\\[1em]
1,&(\delta,\epsilon)=(0,0).
\end{cases}
\end{equation}
Making $\phi(0,0)=1$ is a harmless convention.
The main result of Li and Li~\cite{lili:2005} (see also the corrected arXiv version~\cite{lili:2025arXiv}) is
\begin{subequations}\label{eq:lili2005}
\begin{align}
|\lambda_i(A)-\lambda_i(\wtd A)|
   &{\le \phi(\eta_i,\epsilon_2)\,\epsilon_2}
       \quad\mbox{for $i=1,2,\ldots,N$},  \label{eq:lili2005a} \\
\|\blambda(A)-\blambda(\wtd A)\|_{\infty}
   &{\le \phi(\eta,\epsilon_2)\,\epsilon_2},  \label{eq:lili2005b}
\end{align}
\end{subequations}
where $\|\cdot\|_{\infty}$ is the vector $\ell_{\infty}$-norm.
Inequality \eqref{eq:lili2005b} is actually a corollary
of \eqref{eq:lili2005a} because $\eta_i\ge\eta$ for all $i$ and
$$
\|\blambda(A)-\blambda(\wtd A)\|_{\infty}=\max_{1\le i\le N}|\lambda_i(A)-\lambda_i(\wtd A)|.
$$
The inequalities in \eqref{eq:lili2005} imply both of the following classical bounds.
\begin{enumerate}[(i)]
    \item Weyl's inequality \cite{bhat:1997,parl:1998,stsu:1990}:
     \begin{equation}\label{bd:classical-1}
     \|\blambda(A)-\blambda(\wtd A)\|_{\infty} \le \epsilon_2.
     \end{equation}
    \item Quadratic residual bounds      \cite{bhat:1997,demm:1997,govl:2013,math:1998,parl:1998,stsu:1990}:
    If the spectra of $H_1$ and $H_2$ are disjoint, i.e., $\eta>0$, then
    \begin{equation}\label{bd:classical-2}
    \|\blambda(A)-\blambda(\wtd A)\|_{\infty} \le \frac {\epsilon_2^2}{\eta}.
    \end{equation}
\end{enumerate}
For $\eta>0$, the Li--Li bound is no larger than either classical bound:
$$
\phi(\eta,\epsilon_2)\,\epsilon_2
=\frac{2\epsilon_2^2}{\eta+\sqrt{\eta^2+4\epsilon_2^2}}
\le \min\left\{\epsilon_2,\frac{\epsilon_2^2}{\eta}\right\}.
$$
The bound in \eqref{bd:classical-1} is linear in $\epsilon_2$ and
independent of $\eta$. The quadratic bound in \eqref{bd:classical-2}
can be much smaller than \eqref{bd:classical-1} when $\epsilon_2\ll\eta$, but unfortunately becomes unbounded as
$\eta\to 0$ with $\epsilon_2>0$ fixed. The Li--Li bound retains
the quadratic behavior for small $\epsilon_2/\eta$ and reduces to
$\epsilon_2$ when $\eta=0$.

The inequalities in \eqref{eq:lili2005} provide bounds on
individual differences $|\lambda_i(A)-\lambda_i(\wtd A)|$.
On the other hand, the classical Hoffman--Wielandt theorem \cite{howi:1953} applied to $A$ and $\wtd A$ gives
\begin{equation}\label{eq:howi1953}
\|\blambda(A)-\blambda(\wtd A)\|_2=\sqrt{\sum_{i=1}^N|\lambda_i(A)-\lambda_i(\wtd A)|^2}
  \le\|\what E\|_{\F}=\sqrt 2\|E\|_{\F}=:\sqrt 2\,\epsilon_{\F}
\end{equation}
which bounds the square root of the total sum of the squared differences $|\lambda_i(A)-\lambda_i(\wtd A)|^2$.
More generally, the Lidskii--Mirsky--Wielandt theorem \cite{bhat:1997,stsu:1990} yields
\begin{equation}\label{eq:LMW}
\big\|\diag\big(\blambda(A)-\blambda(\wtd A)\big)\big\|_{\UI}\le\|\what E\|_{\UI}
\end{equation}
for every unitarily invariant norm $\|\cdot\|_{\UI}$ where $\diag(\cdot)$ turns a vector into
a diagonal matrix with the vector's components as the diagonal entries in order.
{It is natural to ask whether the factor in \eqref{eq:lili2005} also improves the Frobenius estimate \eqref{eq:howi1953}, or more generally a unitarily invariant norm of the diagonal matrix of ordered eigenvalue changes. Specifically, we ask whether}
\begin{equation}\label{eq:main}
\big\|\diag\big(\blambda(A)-\blambda(\wtd A)\big)\big\|_{\UI}
\le \phi(\eta,\epsilon_2)\,\|\what E\|_{\UI}
\end{equation}
holds for every unitarily invariant norm $\|\cdot\|_{\UI}$ on $\bbC^{N\times N}$.
We prove that this inequality holds.
{Our argument proves the stronger statement that each of the positive and negative parts of the eigenvalue-change vector is weakly majorized by $\phi(\eta,\epsilon_2)$ times the singular-value vector of $E$, padded with zeros.}
The factor $\phi(\eta,\epsilon_2)$ is sharp.
As an application, the same conclusion is obtained for the singular value problem:
the change in the singular values caused by dropping off-diagonal blocks
obeys an analogous bound for every unitarily invariant norm.

The rest of this paper is organized as follows.
We establish preliminary results in \cref{sec:prelim}.
\Cref{sec:UI} proves \eqref{eq:main} for every unitarily invariant norm.
In \cref{sec:svd}, we apply this result to the singular value perturbation problem.
Finally, conclusions are drawn in \cref{sec:concl}.

{\bf Notation.}
$\bbC^{m\times n}$ denotes the set of $m\times n$ complex
matrices, while $\bbR^{m\times n}$ is the set of $m\times n$ real
matrices.
The $n\times n$ identity matrix is $I_n$, or simply $I$ when its size is clear from the context.
For square matrices $X,Y$ of the same size, $[X,Y]=XY-YX$ is their commutator, 
and  $\tr(X)$ denotes the trace of $X$.
The $j$th largest singular value of a matrix $X$ is denoted by $\sigma_j(X)$. It is noted that we may extend the 
singular-value sequence of a matrix by zeros whenever a longer sequence is needed.
For a real vector
$\bx$, $|\bx|$ denotes its entrywise absolute value, and
$\bx^\downarrow$ denotes its components rearranged in nonincreasing order.

\section{Preliminaries}
\label{sec:prelim}
In this section, we will establish or cite some preliminary results that are necessary for our eventual proof of
our main result \eqref{eq:main} of this paper.

The following lemma, a key tool to establish our main result for \eqref{eq:main}, applies to any Hermitian matrices $\wtd A$ and $\what E$, without requiring the block structure in \cref{sec:intro}. Set
\begin{equation}\label{eq:H(t)}
H(t):=\wtd A+t\what E\quad\mbox{for $t\in[0,1]$},
\end{equation}
and denote by
\begin{subequations}\label{eq:eig4H(t)}
\begin{equation}\label{eq:eig4H(t)-1}
\lambda_1(t)\ge\lambda_2(t)\ge\cdots\ge\lambda_N(t)
\end{equation}
the ordered eigenvalues of $H(t)$, and
\begin{equation}\label{eq:blambda}
\blambda(t)=[\lambda_1(t),\lambda_2(t),\ldots,\lambda_N(t)]^{\T}\in\bbR^N.
\end{equation}
\end{subequations}

\begin{lemma}
\label{lem:derivative}
Let $\wtd A,\,\what E\in\bbC^{N\times N}$ be Hermitian, and accordingly $H(t)$ and its eigenvalues
$\lambda_i(t)$ as in \eqref{eq:H(t)} and \eqref{eq:eig4H(t)}.
Then each map $t\mapsto\lambda_i(t)$ is Lipschitz continuous, and at every point $t\in(0,1)$ at which all the ordered eigenvalue functions
$\lambda_1(\cdot),\ldots,\lambda_N(\cdot)$ are differentiable (hence for almost every $t\in[0,1]$),
there is an orthonormal set $\{\bx_1(t),\ldots,\bx_N(t)\}$ of eigenvectors of
$H(t)$, associated with its eigenvalues $\lambda_i(t)$, i.e.,  $H(t)\bx_i(t)=\lambda_i(t)\bx_i(t)$, satisfying
\begin{equation}
\label{eq:deriv-l2}
\frac{\rmd}{\rmd\!t}\lambda_j(t)=
\bx_j(t)^*\what E\bx_j(t)=:z_j(t)
\end{equation}
for $j=1,2,\ldots,N$.
\end{lemma}

\begin{proof}
Lipschitz continuity of ordered Hermitian eigenvalues under Hermitian
perturbations is standard: Weyl's inequality \cite{stsu:1990} (see, also, \eqref{bd:classical-1}) yields
$$
\bigl|\lambda_j(t+h)-\lambda_j(t)\bigr|
\le
\|h\what E\|_2
=
|h|\,\|\what E\|_2,
$$
so each $\lambda_j(t)$ is Lipschitz continuous, hence absolutely
continuous and differentiable almost everywhere
\cite[Lemma~6.1.3 and Corollary~6.1.5]{heil:2019}.
The intersection of the corresponding full-measure differentiability sets still has full measure.
Fix a point $t$ in this intersection, and construct the basis $\{\bx_j(t)\}$ as follows.
{Let $\nu$ run over the distinct eigenvalues of $H(t)$, and let}
$\mathcal{E}_{\nu}:=\ker\bigl(H(t)-\nu I\bigr)$ be the corresponding eigenspace,
of dimension $n_{\nu}$. Write $P_{\nu}$ for the orthogonal projector onto
$\mathcal{E}_{\nu}$. The \emph{compression} of $\what E$ to $\mathcal{E}_{\nu}$ is the
Hermitian operator
$$
\what E_{\nu}:=
P_{\nu} \what E\big|_{\mathcal{E}_{\nu}}:
\mathcal{E}_{\nu}\to\mathcal{E}_{\nu}.
$$
(Equivalently, if $X_{\nu}$ is any orthonormal basis matrix of $\mathcal{E}_{\nu}$,
then $X_{\nu}^*\what EX_{\nu}$ is an $n_{\nu}\times n_{\nu}$ Hermitian matrix representing
$\what E_{\nu}$.)
Hence $\what E_{\nu}$ admits an orthonormal eigenbasis of $\mathcal{E}_{\nu}$.
Doing this for every $\nu$ and concatenating yields an orthonormal basis
$\{\bx_1(t),\ldots,\bx_N(t)\}$ of $\bbC^N$ consisting of eigenvectors of $H(t)$
such that each $\bx_j(t)$ is also an eigenvector of the compression of $\what E$ on
its eigenspace. In particular $\bx_j(t)^*\what E\bx_k(t)=0$ whenever $\bx_j(t),\bx_k(t)$
lie in the same eigenspace and $j\neq k$.

To identify the derivatives within each eigenspace, suppose the indices
associated with $\nu$ are $\ell,\ldots,\ell+n_\nu-1$, and let
$\theta_1\ge\cdots\ge\theta_{n_\nu}$ be the eigenvalues of $\what E_\nu$.
Apply \cite[Chapter~II, \S5.4, Theorem~5.4, Eq.~(5.10), and Remark~5.5, pp.~111--112]{kato:1995} to $T(h)=H(t)+h\what E$. Since $H(t)$ is Hermitian, $\nu$ is semisimple, and the first-order coefficients are the eigenvalues of the compression $\what E_\nu=P_\nu T'(0)|_{\mathcal E_\nu}$. Ordering these coefficients in decreasing order gives the right-sided expansion
$$
 \lambda_{\ell+j-1}(t+h)=\nu+h\theta_j+o(h),
 \quad\mbox{for $1\le j\le n_\nu$},
$$
as $h\downarrow0$.
Since the ordered eigenvalue functions are differentiable at $t$, their derivatives equal these right derivatives:
$\lambda_{\ell+j-1}'(t)=\theta_j$. Indexing the selected eigenvectors
of $\what E_\nu$ in this order proves \eqref{eq:deriv-l2}.
\end{proof}

Absolute continuity also gives the integral identity \cite[Theorem~6.4.2]{heil:2019}:
$$
\lambda_j(1)-\lambda_j(0)=\int_0^1\lambda_j'(t)\,\rmd\!t.
$$
For simple Hermitian eigenvalues, \eqref{eq:deriv-l2} is the usual
Rayleigh-quotient derivative formula. For a simple eigenvalue of a
general analytic matrix-valued function $K(t)$, the corresponding
formula is $\lambda'(t)=\by(t)^*K'(t)\bx(t)$, where $\bx(t)$ and
$\by(t)$ are right and left eigenvectors normalized by
$\by(t)^*\bx(t)=1$ \cite{grlo:2020,part:1978}.

The following lemma, due to Davis-Kahan~\cite{daka:1970},  is well-known (see, also, \cite[Chapter~V, Theorem~3.1, pp.~247--248]{stsu:1990}).
\begin{lemma}[{Davis--Kahan~\cite{daka:1970}}]
\label{lem:gap}
For any
$X\in\bbC^{m\times n}$,
$$
\|H_1X-XH_2\|_{\F}\ge\eta\|X\|_{\F},
$$
where $\eta=\gap\bigl(\eig(H_1),\eig(H_2)\bigr)$ as in \eqref{eq:gaps-2}.
\end{lemma}

Von Neumann~\cite{neum:1937} identified unitarily invariant norms with symmetric gauge functions.
There is a one-to-one correspondence between the two.
We now state the definition and this correspondence (see also \cite[Section~II.3]{stsu:1990}).

\begin{definition}\label{def:SGF}
A function $\Phi\,:\,\bbR^N\to [0,\infty)$ is a symmetric gauge function
if it satisfies the following conditions.
\begin{enumerate}[(1)]
  \item $\bx\ne \mathbf{0}\quad\Leftrightarrow\quad \Phi(\bx)>0$;
  \item $\Phi(\alpha\bx)=|\alpha|\Phi(\bx)$ for any $\alpha\in\bbR$;
  \item $\Phi(\bx+\by)\le\Phi(\bx)+\Phi(\by)$;
  \item $\Phi(P\bx)=\Phi(\bx)$ for any permutation matrix $P\in\bbR^{N\times N}$;
  \item $\Phi(|\bx|)=\Phi(\bx)$ where $|\bx|$ takes entrywise absolute value on $\bx$.
\end{enumerate}
\end{definition}

A nonnegative vector $\bx\in\bbR^N$ is said to be \emph{weakly majorized} by another nonnegative vector
$\by\in\bbR^N$, written $\bx\prec_w\by$, if the sum of the $k$ largest components of $\bx$ does not
exceed that of $\by$ for each $k=1,2,\ldots,N$.
Fan's dominance theorem~\cite[Theorem~IV.2.2]{bhat:1997} asserts that
$\Phi_{\UI}(\bx)\le\Phi_{\UI}(\by)$ for every symmetric gauge function $\Phi_{\UI}$ if and only if
$\bx\prec_w\by$.

Nonnegative vectors of different lengths are compared by weak majorization after appending zeros to make their lengths equal.

The singular values of a matrix $X\in\bbC^{N\times N}$ are denoted by
{$\sigma_1(X)\ge\cdots\ge\sigma_N(X)$, arranged in nonincreasing order.}
For a given unitarily invariant norm $\|\cdot\|_{\UI}$ on $\bbC^{N\times N}$, there exists a
symmetric gauge function $\Phi_{\UI}\,:\,\bbR^N\to [0,\infty)$ such that
$$
\|X\|_{\UI}=\Phi_{\UI}(\bsigma(X))\quad\mbox{for $X\in\bbC^{N\times N}$}
$$
where $\bsigma(X)\in\bbR^N$ is the vector of singular values of \(X\) arranged in nonincreasing order; conversely,
for a given symmetric gauge function $\Phi\,:\,\bbR^N\to [0,\infty)$,
$$
\|X\|_{\Phi}:=\Phi(\bsigma(X))\quad\mbox{for $X\in\bbC^{N\times N}$}
$$
is a unitarily invariant norm on $\bbC^{N\times N}$. 

For $1\le p<\infty$, the Schatten $p$-norm is
$$
\|X\|_{S_p}=\Biggl(\sum_{i=1}^N\bigl[\sigma_i(X)\bigr]^p\Biggr)^{1/p},
$$
while $\|X\|_{S_\infty}=\lim_{p\to\infty}\|X\|_{S_p}=\sigma_1(X)=\|X\|_2$.
In particular $\|X\|_{\F}=\|X\|_{S_2}$, and the trace norm, also known as the nuclear norm, is
$\|X\|_{S_1}=\sum_{i=1}^N\sigma_i(X)$.
For compatible matrices $X\in\bbC^{m\times n}$ and $Y\in\bbC^{n\times k}$, the Schatten H\"older inequality is
\begin{equation}
\label{eq:holder}
\|XY\|_{S_1}\le\|X\|_{S_p}\,\|Y\|_{S_q},
\end{equation}
where $1\le p,\, q\le\infty$ such that $1/p+1/q=1$ ($1/\infty=0$ by convention).
To prove it, pad $X$ and $Y$ with zero rows and/or columns to square matrices $\what X,\,\what Y\in\bbC^{N\times N}$
$N\ge \max\{m,n,k\}$ such that $XY$ sits at the top-left corner of $\what X\what Y$, 
so all the relevant Schatten norms are preserved. A polar decomposition of the embedded product gives a unitary matrix $U$ such that $XYU$ is positive semidefinite. Von Neumann's trace inequality \cite{neum:1937} (see also \cite[p.~183]{hojo:1991}) then gives
\begin{align*}
\|XY\|_{S_1}
=\|\widehat X\widehat Y\|_{S_1}
=\tr(\widehat X\widehat YU)
\le\sum_i\sigma_i(\widehat X)\sigma_i(\widehat YU)
\le\|X\|_{S_p}\|Y\|_{S_q}
\end{align*}
as expected.

\begin{lemma}
\label{lem:trace-commutator}
Let $W\ne0$ and $B$ be two $N\times N$ Hermitian matrices.
Then
\begin{equation}\label{eq:trace-commutator}
\|[B,W]\|_{\F}^2
\le 4\|W\|_2\left(
 \|W\|_{S_1}\|B\|_2^2
 -\frac{|\tr(BW)|^2}{\|W\|_{S_1}}\right).
\end{equation}
\end{lemma}

\begin{proof}
Without loss of generality, we may assume that $W$ is diagonal and write
$W=\diag(\omega_1,\ldots,\omega_N)$ and $B\equiv [b_{ij}]$.
Since $|\omega_i|\le \|W\|_2$,
$$
 |\omega_i-\omega_j|^2
 \le2(|\omega_i|^2+|\omega_j|^2)
 \le2\|W\|_2(|\omega_i|+|\omega_j|).
$$
With diagonal $W$, the diagonal entries in the commutator $[B,W]$ vanish. Using $B=B^*$ gives
\begin{align}
\|[B,W]\|_{\F}^2
 &=\sum_{i\ne j}|\omega_i-\omega_j|^2|b_{ij}|^2\nonumber\\
 &\le4\|W\|_2\sum_i |\omega_i|\sum_{j\ne i}|b_{ij}|^2\nonumber\\
 &=4\|W\|_2\left(\tr(|W|B^2)
                   -\sum_i |\omega_i||b_{ii}|^2\right),
\label{eq:trace-commutator-diagonal}
\end{align}
where $|W|=\diag(|\omega_1|,\ldots,|\omega_N|)$.
The weighted Cauchy--Schwarz inequality yields
\begin{align}
 |\tr(BW)|^2
  &=\left|\sum_i \omega_i b_{ii}\right|^2 \nonumber\\
  &\le\left(\sum_i|\omega_i|\right)
       \left(\sum_i |\omega_i||b_{ii}|^2\right) \nonumber\\
  &=\|W\|_{S_1}\sum_i |\omega_i||b_{ii}|^2, \nonumber \\
\Rightarrow\quad \sum_i |\omega_i||b_{ii}|^2&\ge \frac{|\tr(BW)|^2}{\|W\|_{S_1}}.  \label{eq:trace-commutator-CS}
\end{align}
The trace-norm H\"older inequality gives
\begin{equation}\label{eq:trace-commutator-Holder}
 \tr(|W|B^2)
 \le\|W\|_{S_1}\|B\|_2^2.
\end{equation}
Combining \eqref{eq:trace-commutator-diagonal},
\eqref{eq:trace-commutator-CS}, and
\eqref{eq:trace-commutator-Holder} proves the result.
\end{proof}

\begin{lemma}
\label{lem:UI-flat-layer}
Let $W$ be Hermitian with $r$ eigenvalues equal to $1$, $r$ equal to
$-1$, and all remaining eigenvalues zero, {where $r\ge1$. Let $P$ be an orthogonal projector of rank $k\ge1$ and of the same size as $W$.}
Then
\begin{equation}\label{eq:UI-flat-layer}
 \|[P,W]\|_{\F}^2
 +\frac{2}{\ell}|\tr(PW)|^2\le2\ell,
\end{equation}
where $\ell=\min\{k,r\}$.
\end{lemma}

\begin{proof}
Since $P$ is an orthogonal projector, we know $P^2=P$ and $P$ is Hermitian too. Using $\tr(XY)=\tr(YX)$ we get
\begin{align}
\|[P,W]\|_{\F}^2
 &=\tr\big((PW-WP)^*(PW-WP)\big) \nonumber\\
 &=\tr\big((WP-PW)(PW-WP)\big) \nonumber\\
 &=\tr\big(WPW-WPWP-PWPW+PW^2P\big) \nonumber\\
 &=2\tr(PW^2)-2\tr(WPWP)  \nonumber\\
 &=2\tr(PW^2)-2\tr\bigl((PWP)^2\bigr), \label{eq:UI-flat-layer:pf-1}
\end{align}
{where we have used $\tr(WPWP)=\tr(WP^2WP^2)=\tr(PWP^2WP)=\tr\bigl((PWP)^2\bigr)$.}
Consider now $k\le r$. Use $W^2\preceq I$ (i.e., $I-W^2$ is positive semidefinite) to obtain
\begin{equation}\label{eq:UI-flat-layer:pf-2}
\tr(PW^2)=\tr(P^2W^2)=\tr(PW^2P)\le\tr(PP)=\tr(P)=k,
\end{equation}
{and use the Cauchy--Schwarz inequality to get}
$$
|\tr(PW)|^2=|\tr(P^3W)|^2
   =|\tr\big(P(PWP)\big)|^2
  \le\tr(P^2)\tr\big((PWP)^2\big)
  =k\tr\big((PWP)^2\big),
$$
yielding
\begin{equation}\label{eq:UI-flat-layer:pf-3}
\frac 1k|\tr(PW)|^2\le \tr\big((PWP)^2\big).
\end{equation}
Combine \eqref{eq:UI-flat-layer:pf-1}, \eqref{eq:UI-flat-layer:pf-2} and \eqref{eq:UI-flat-layer:pf-3} to get
$$
\|[P,W]\|_{\F}^2+\frac{2}{k}|\tr(PW)|^2
 \le\|[P,W]\|_{\F}^2+2\tr\big((PWP)^2\big)
 =2\tr(PW^2)
 \le 2k,
$$
which proves \eqref{eq:UI-flat-layer} for the case $k\le r$.

If $k>r$, apply \Cref{lem:trace-commutator} with $B=2P-I$.
Then $\|B\|_2=1$, $\|W\|_2=1$, $\|W\|_{S_1}=2r$, and $\tr(W)=0$, so
\[
\|[B,W]\|_{\F}^2
\le 4\left(2r-\frac{|\tr(BW)|^2}{2r}\right)
=8r-\frac{2}{r}|\tr(BW)|^2.
\]
Since $[B,W]=2[P,W]$ and $\tr(BW)=2\tr(PW)$,
\begin{align*}
\|[P,W]\|_{\F}^2
&=\frac14\|[B,W]\|_{\F}^2
\le\frac14\left(8r-\frac{2}{r}|\tr(BW)|^2\right)\\
&=2r-\frac{1}{2r}|\tr(BW)|^2
=2r-\frac{2}{r}|\tr(PW)|^2,
\end{align*}
a simple rearrangement of which leads to
\eqref{eq:UI-flat-layer} for the case $k> r$.
\end{proof}

\section{Main Result}
\label{sec:UI}

We now prove \eqref{eq:main} for every unitarily invariant norm.
The argument refines a trace-norm commutator estimate by taking the rank of a projector into account, and then integrates a singular-value threshold decomposition of $\what E$.
 
Throughout the section, $A$, $\wtd A$, and the associated block notation are as in \eqref{eq:A}--\eqref{eq:whatE}.
Define $H(t)$ as in \eqref{eq:H(t)}, and adopt the corresponding notation
in \eqref{eq:eig4H(t)}. It can be seen that
$$
H(0)=\wtd A, \quad \blambda(0)=\blambda(\wtd A); \quad
H(1)=A, \quad \blambda(1)=\blambda(A).
$$
%
Let $r=\rank(E)$, pad the singular values of $E$ with zeros, and set
\begin{equation}\label{eq:UI-prefix}
 s_0=0,\quad s_k=\sum_{j=1}^k\sigma_j(E),\quad 1\le k\le N,
 \quad
 \bsigma_E=[\sigma_1(E),\ldots,\sigma_r(E),0,\ldots,0]^{\T}\in\bbR^N.
\end{equation}
Thus $s_k-s_{k-1}=\sigma_k(E)$ for $k\ge1$, with the convention $\sigma_k(E)=0$ for $k>r$.

For a vector $\bx\equiv [x_j]\in\bbR^N$, write $\bx_+$ and $\bx_-$ for the vectors whose components are $(x_j)_+=\max\{x_j,0\}$ and $(x_j)_-=\max\{-x_j,0\}$.

\begin{theorem}\label{thm:main-UI}
Let $A$, $\wtd A$, $\what E$, $\eta$, and $\epsilon_2$ be as in
\cref{sec:intro}, and let
{$\Delta\blambda=\blambda(A)-\blambda(\wtd A)$, with components $\Delta\lambda_j:=\lambda_j(A)-\lambda_j(\wtd A)$.}
Then
\begin{equation}\label{eq:UI-one-sided-majorization}
 (\Delta\blambda)_+\prec_w\phi(\eta,\epsilon_2)\bsigma_E,
 \quad
 (\Delta\blambda)_-\prec_w\phi(\eta,\epsilon_2)\bsigma_E,
\end{equation}
where $\phi(\cdot,\cdot)$ is as defined  in \eqref{eq:phi}.
Consequently,
\begin{equation}\label{eq:UI-full-majorization}
 |\Delta\blambda|\prec_w\phi(\eta,\epsilon_2)\bsigma(\what E),
\end{equation}
{and \eqref{eq:main} holds for every unitarily invariant norm. In particular,}
$$
\|\Delta\blambda\|_2\le\sqrt2\,\phi(\eta,\epsilon_2)\|E\|_{\F}.
$$
The factor $\phi(\eta,\epsilon_2)$ is sharp.
\end{theorem}

The proof of \Cref{thm:main-UI} needs the following technical lemmas.
A well-known relationship (see e.g., \cite[Section 9.4.3]{govl:2013}) between the eigenvalues of $\what E$ in~\eqref{eq:whatE} and singular values of $E$ gives the following result.

\begin{lemma}
\label{lem:sV}
For $\what E$ as in~\eqref{eq:whatE},
the nonzero singular values of $\what E$ are those of $E$, each repeated twice.
In particular,
\[
\|\what E\|_2=\|E\|_2=:\epsilon_2,
\quad
\|\what E\|_{\F}=\sqrt2\,\|E\|_{\F}=:\sqrt2\,\epsilon_{\F}.
\]
\end{lemma}

\begin{lemma}
\label{lem:UI-prefix-commutator}
Assume $E\ne0$ in~\eqref{eq:whatE}, so $s_k>0$ for $1\le k\le N$ where $s_k$ is given in \eqref{eq:UI-prefix}. For every orthogonal projector $P$ of rank
$k\ge1$,
\begin{equation}\label{eq:UI-prefix-commutator}
 \|[P,\what E]\|_{\F}^2
 \le2\epsilon_2 s_k
       -\frac{2\epsilon_2}{s_k}|\tr(P\what E)|^2.
\end{equation}
\end{lemma}

\begin{proof}
Write a singular value decomposition
$E=\sum_{j=1}^r\sigma_j(E)\bb_j\ba_j^*$, with orthonormal
families $\{\ba_j\}$ and $\{\bb_j\}$, and embed them as
\[
 \widehat\ba_j=\begin{bmatrix}\ba_j\\0\end{bmatrix}\in\bbC^N,
 \quad
 \widehat\bb_j=\begin{bmatrix}0\\\bb_j\end{bmatrix}\in\bbC^N.
\]
It is noted that $r=\rank(E)$ and $\sigma_j(E)=0$ for $j>r$. For $0\le\tau<\epsilon_2$, define
\begin{align*}
W_\tau&=\sum_{j=1}^r\mathbf{1}_{\{0\le\tau<\sigma_j(E)\}}
       \bigl(\widehat\ba_j\widehat\bb_j^*
              +\widehat\bb_j\widehat\ba_j^*\bigr),\\
r_\tau&=\sum_{j=1}^r\mathbf{1}_{\{0\le\tau<\sigma_j(E)\}}, \\
\ell_\tau&=\min\{k,r_\tau\}=\sum_{j=1}^{\min\{k,r\}}\mathbf{1}_{\{0\le\tau<\sigma_j(E)\}},
\end{align*}
where $\mathbf{1}_{\{0\le\tau<\sigma_j(E)\}}$ equals $1$ when $0\le\tau<\sigma_j(E)$ and $0$ otherwise. The matrix $W_\tau$ has
$r_\tau$ eigenvalues $1$, $r_\tau$ eigenvalues $-1$, and all remaining eigenvalues zero.
All threshold integrals below use Lebesgue measure $\rmd\!\tau$; the integrands are finite step functions and,
in the case of $W_\tau$, it is matrix-valued and the integration is done entrywise. The value at $\tau=\epsilon_2$ can be set to $0$ without changing any integral. Since
\[
 \int_0^{\epsilon_2}\mathbf{1}_{\{0\le\tau<\sigma_j(E)\}}\,\rmd\!\tau=\sigma_j(E),
\]
we have
\begin{equation}\label{eq:UI-threshold-integrals}
 \what E=\int_0^{\epsilon_2}W_\tau\,\rmd\!\tau,
 \qquad s_k=\int_0^{\epsilon_2}\ell_\tau\,\rmd\!\tau.
\end{equation}
Because $E\ne0$ and $\sigma_1(E)=\epsilon_2$, we have $\ell_\tau\ge1$ for $0\le\tau<\epsilon_2$.
Set $\vartheta_\tau=\tr(PW_\tau)$.
The triangle inequality and the Cauchy--Schwarz inequality give
\begin{align*}
 \|[P,\what E]\|_{\F}^2
 &\le\left(\int_0^{\epsilon_2}\|[P,W_\tau]\|_{\F}\,\rmd\!\tau\right)^2\\
 &\le\epsilon_2\int_0^{\epsilon_2}\|[P,W_\tau]\|_{\F}^2\,\rmd\!\tau,\\
 |\tr(P\what E)|^2
 &=\left|\int_0^{\epsilon_2}\tr(PW_\tau)\,\rmd\!\tau\right |^2 \\
 &=\left|\int_0^{\epsilon_2}\sqrt{\ell_\tau}\,
              \frac{\vartheta_\tau}{\sqrt{\ell_\tau}}\,\rmd\!\tau\right|^2\\
 &\le\left(\int_0^{\epsilon_2}\ell_\tau\,\rmd\!\tau\right)
       \left(\int_0^{\epsilon_2}\frac{|\vartheta_\tau|^2}{\ell_\tau}\,\rmd\!\tau\right)\\
 &=s_k\int_0^{\epsilon_2}\frac{|\vartheta_\tau|^2}{\ell_\tau}\,\rmd\!\tau.
\end{align*}
The quotients above are defined for $0\le\tau<\epsilon_2$,
where $\ell_\tau\ge1$. By \Cref{lem:UI-flat-layer},
$|\vartheta_\tau|^2/\ell_\tau\le\ell_\tau\le\min\{k,r\}$.
At $\tau=\epsilon_2$, we can assign these quotients the value zero which   does not change the integrals, as a single point
has Lebesgue measure zero.
Applying \Cref{lem:UI-flat-layer} at each threshold now yields
\begin{align*}
 \|[P,\what E]\|_{\F}^2+\frac{2\epsilon_2}{s_k} |\tr(P\what E)|^2
 &\le\epsilon_2\int_0^{\epsilon_2}
       \left(\|[P,W_\tau]\|_{\F}^2+
                 \frac{2|\vartheta_\tau|^2}{\ell_\tau}\right)\,\rmd\!\tau\\
 &\le2\epsilon_2\int_0^{\epsilon_2}\ell_\tau\,\rmd\!\tau
 =2\epsilon_2s_k,
\end{align*}
which proves \eqref{eq:UI-prefix-commutator}.
\end{proof}


\begin{lemma}
\label{lem:UI-projector-trace}
If an orthogonal projector $P$ of rank $k\ge1$ commutes with
$H(t)=\wtd A+t\what E$, then
\begin{equation}\label{eq:UI-projector-trace}
 |\tr(P\what E)|\le c(t)s_k,
\end{equation}
where $s_k$ is in \eqref{eq:UI-prefix} and $c(t)$ is given by
\begin{equation}
\label{eq:c}
c(t)
:=
\begin{cases}
1, & \eta=0,\\[0.4em]
\dfrac{2t\epsilon_2}{\sqrt{\eta^2+4t^2\epsilon_2^2}},
   & \eta>0,
\end{cases}
\quad \mbox{for $t\in[0,1]$}.
\end{equation}
\end{lemma}

\begin{proof}
If $E=0$, the conclusion is immediate. Suppose $E\ne0$. When $\eta=0$, \Cref{lem:UI-prefix-commutator} and nonnegativity of the commutator norm give
\[
 0\le2\epsilon_2s_k-\frac{2\epsilon_2}{s_k}|\tr(P\what E)|^2,
 \qquad |\tr(P\what E)|\le s_k.
\]
This is \eqref{eq:UI-projector-trace} because $c(t)=1$.
It remains to consider $\eta>0$.
Partition $P$ conformally with $\wtd A$, and let $P_{\mathrm{off}}$ denote its off-diagonal part:
$$
P=\begin{bmatrix}
    P_{11} & P_{12} \\
    P_{21} & P_{22}
  \end{bmatrix}, \quad P_{\mathrm{off}}=\begin{bmatrix}
    0 & P_{12} \\
    P_{21} & 0
  \end{bmatrix}.
$$
Let $\theta=\tr(P\what E)$.
Since $\rank(P_{12})\le\rank(P)=k$, use von Neumann's trace inequality
and the Cauchy--Schwarz inequality to get
\begin{align}
 |\theta|=|\tr(P\what E)|
 &=2|\Rea\big(\tr(P_{12}E)\big)| \nonumber \\
 &\le 2\sum_{j=1}^k\sigma_j(P_{12})\sigma_j(E) \nonumber\\
 &\le 2\|P_{12}\|_{\F}\Big(\sum_{j=1}^k\sigma_j(E)^2\Big)^{1/2} \nonumber\\
 &\le 2\sqrt{\epsilon_2 s_k}\,\|P_{12}\|_{\F}, \label{eq:UI-projector-trace:pf-1}
\end{align}
where for the last inequality we have used
$\sigma_j(E)\le\epsilon_2$ and hence $\sigma_j(E)^2\le\epsilon_2\sigma_j(E)$.
Notice $\|P_{\mathrm{off}}\|_{\F}^2=2\|P_{12}\|_{\F}^2$ since $P_{21}=P_{12}^*$. Thus \eqref{eq:UI-projector-trace:pf-1} gives
$$
|\theta|^2\le 4\epsilon_2 s_k\|P_{12}\|_{\F}^2 = 2\epsilon_2 s_k\|P_{\mathrm{off}}\|_{\F}^2.
$$
Next we expand $[\wtd A,P]$ as
$$
[\wtd A,P]=\wtd AP-P\wtd A=\begin{bmatrix}
                             H_1P_{11}-P_{11}H_1 & H_1P_{12}-P_{12}H_2 \\
                             H_2P_{21}-P_{21}H_1 & H_2P_{22}-P_{22}H_2
                           \end{bmatrix},
$$
which, together with $P_{21}=P_{12}^*$ and $H_2P_{21}-P_{21}H_1=-(H_1P_{12}-P_{12}H_2)^*$, gives
\begin{align*}
 \|[\wtd A,P]\|_{\F}^2
 &=\|[H_1,P_{11}]\|_{\F}^2+\|[H_2,P_{22}]\|_{\F}^2
       +2\|H_1P_{12}-P_{12}H_2\|_{\F}^2\\
 &\ge 2\|H_1P_{12}-P_{12}H_2\|_{\F}^2 \\
 &\ge  \eta^2\|P_{\mathrm{off}}\|_{\F}^2. \qquad (\mbox{by \Cref{lem:gap}})
\end{align*}
By the assumption that $P$ commutes with $H(t)=\wtd A+t\what E$, we have $[\wtd A,P]=-t[\what E,P]$, and hence
\begin{equation}\label{eq:UI-commutant-gap}
 \eta^2\|P_{\mathrm{off}}\|_{\F}^2
 \le\|[\wtd A,P]\|_{\F}^2
 =t^2\|[P,\what E]\|_{\F}^2.
\end{equation}
Combining \eqref{eq:UI-commutant-gap} with \Cref{lem:UI-prefix-commutator} yields
\begin{align*}
 \eta^2|\theta|^2
 &\le2t^2\epsilon_2 s_k\|[P,\what E]\|_{\F}^2\\
 &\le4t^2\epsilon_2^2(s_k^2-|\theta|^2),
\end{align*}
a simple rearrangement of which leads to \eqref{eq:UI-projector-trace}.
\end{proof}

A simple computation gives, for all $\eta,\epsilon_2\ge0$,
\begin{equation}
\label{eq:intc}
\int_0^1 c(t)\,\rmd\! t
=
\phi(\eta,\epsilon_2),
\end{equation}
{where $\phi$ is defined in \eqref{eq:phi}.}
We can now prove \Cref{thm:main-UI}.

\begin{proof}[Proof of Theorem~\ref{thm:main-UI}]
For any $1\le k\le N$, let $\bbI_+$ be the set of indices of the strictly positive
components among the $k$ largest components of
$(\Delta\blambda)_+$.
If $\bbI_+$ is empty, 
the desired inequality is immediate. 
Otherwise, at almost every $t$, choose the eigenbasis from \Cref{lem:derivative} and form
$P_{\bbI_+}(t)=\sum_{j\in\bbI_+}\bx_j(t)\bx_j(t)^*$.
This projector has rank $1\le|\bbI_+|\le k$ and commutes with $H(t)$.
Its scalar trace $\tr(P_{\bbI_+}(t)\what E)$ equals
$\sum_{j\in\bbI_+}\lambda_j'(t)$, so it is measurable and
integrable; it is defined to be zero on the exceptional null set.
Since each ordered eigenvalue is absolutely continuous,
\Cref{lem:UI-projector-trace} and
\eqref{eq:intc} give
\begin{align*}
 \sum_{j=1}^k\bigl((\Delta\blambda)_+\bigr)_j^\downarrow
 =\sum_{j\in \bbI_+}\Delta\lambda_j
 &=\int_0^1\sum_{j\in \bbI_+}\lambda_j'(t)\,\rmd\!t\\
 &=\int_0^1\tr\bigl(P_{\bbI_+}(t)\what E\bigr)\,\rmd\!t
 \le\int_0^1c(t)s_{|\bbI_+|}\,\rmd\!t
 \le\phi(\eta,\epsilon_2)s_k.
\end{align*}
The last step uses $|\bbI_+|\le k$ and the fact that $k\mapsto s_k$ is nondecreasing.
For the negative part, let $\bbI_-$ be the indices of the strictly
positive components among the $k$ largest components of
$(\Delta\blambda)_-$; these are indices for which $\Delta\lambda_j<0$.
If $\bbI_-$ is empty, the desired sum is zero. Otherwise, the
same construction gives
$$
 \sum_{j=1}^k\bigl((\Delta\blambda)_-\bigr)_j^\downarrow
 =\sum_{j\in\bbI_-}(-\Delta\lambda_j)
=-\int_0^1\tr\bigl(P_{\bbI_-}(t)\what E\bigr)\,\rmd\!t.
$$
{\Cref{lem:UI-projector-trace} bounds the absolute value of the trace, so the integrand does not exceed $c(t)s_{|\bbI_-|}$.}
Integrating as above yields the negative-part estimate.
The index sets are selected from the endpoint differences, while the
eigenvalues along the path retain their original decreasing order.
This establishes \eqref{eq:UI-one-sided-majorization}.

Among the $k$ largest absolute eigenvalue changes, let $k_+$ be
the number of positive changes and $k_-$ the number of negative
changes, so $k_++k_-\le k$. The sum of any $k_+$ positive components is at most the sum of the $k_+$ largest positive components, and likewise for the negative components. Therefore
\begin{equation}\label{eq:UI-ab}
 \sum_{j=1}^k|\Delta\blambda|_j^\downarrow
 \le\phi(\eta,\epsilon_2)(s_{k_+}+s_{k_-}).
\end{equation}
The increments $s_j-s_{j-1}=\sigma_j(E)$ are nonnegative and nonincreasing.
If $k_+\ge k_-+2$, then
$$
s_{k_+}+s_{k_-}-(s_{k_+-1}+s_{k_-+1})=\sigma_{k_+}(E)-\sigma_{k_-+1}(E)\le0,
$$
and likewise if $k_-\ge k_++2$ then
$$
s_{k_+}+s_{k_-}-(s_{k_++1}+s_{k_--1})=\sigma_{k_-}(E)-\sigma_{k_++1}(E)\le0,
$$
so, for a fixed sum $k_++k_-$, the quantity $s_{k_+}+s_{k_-}$ is maximized when $|k_+-k_-|\le1$.
Increasing $k_++k_-$ cannot decrease this maximum, because the increments are nonnegative.
Hence the maximum of $s_{k_+}+s_{k_-}$ over $k_++k_-\le k$ is $s_{\lfloor k/2\rfloor}+s_{\lceil k/2\rceil}$, and thus
\begin{equation}\label{eq:UI-balanced}
 \sum_{j=1}^k|\Delta\blambda|_j^\downarrow
 \le\phi(\eta,\epsilon_2)
       \left(s_{\lfloor k/2\rfloor}+s_{\lceil k/2\rceil}\right).
\end{equation}
By \Cref{lem:sV}, $\sigma_{2i-1}(\what E)=\sigma_{2i}(\what E)=\sigma_i(E)$,
with zero padding as needed.
Thus if $k=2\ell$, then $\sum_{j=1}^k\sigma_j(\what E)=2s_\ell$, while if $k=2\ell+1$, the same sum equals $s_\ell+s_{\ell+1}$.
In either case,
\[
s_{\lfloor k/2\rfloor}+s_{\lceil k/2\rceil}
=\sum_{j=1}^k\sigma_j(\what E),
\]
and \eqref{eq:UI-balanced} is \eqref{eq:UI-full-majorization}.
Fan's dominance theorem gives \eqref{eq:main} for every unitarily invariant norm.
The $2\times2$ matrix in \Cref{ex:sharp} attains equality for every
such norm, proving sharpness.
\end{proof}

\begin{example}
\label{ex:sharp}
Take $m=n=1$, $H_1=[\alpha]$, $H_2=[\beta]$ with $\alpha-\beta=\eta\ge 0$, and $E=[\epsilon]$ with
$\epsilon>0$.
Then
\[
A
=
\begin{bmatrix}\alpha & \epsilon \\ \epsilon&\beta\end{bmatrix},
\quad
\wtd A=\begin{bmatrix}\alpha & 0 \\ 0 &\beta\end{bmatrix},
\]
and the eigenvalues of $A$ are
\[
\lambda_1(A)=\frac{\alpha+\beta+\sqrt{\eta^2+4\epsilon^2}}{2},
\quad
\lambda_2(A)=\frac{\alpha+\beta-\sqrt{\eta^2+4\epsilon^2}}{2},
\]
while $\lambda_1(\wtd A)=\alpha$ and $\lambda_2(\wtd A)=\beta$.
A short computation (as in~\cite[Example~1]{lili:2005}) gives absolute differences
\[
\bigl|\lambda_j(A)-\lambda_j(\wtd A)\bigr|
=
\frac {2\epsilon}{\sqrt{\eta^2+4\epsilon^2}+\eta}\cdot \epsilon
\]
for $j=1,2$.
Hence, for every symmetric gauge function $\Phi_{\UI}$,
\begin{equation}\label{eq:sharp-UI}
\begin{aligned}
 \Phi_{\UI}(\blambda(A)-\blambda(\wtd A))
 =\phi(\eta,\epsilon)\Phi_{\UI}(\epsilon,\epsilon)
 =\phi(\eta,\epsilon)
       \left\|\begin{bmatrix}0&\epsilon\\\epsilon&0\end{bmatrix}\right\|_{\UI}.
\end{aligned}
\end{equation}
Equality in \eqref{eq:main} for every unitarily invariant norm holds for this example.
The same equality holds in every larger dimension.
Append eigenvalues to $H_1$ or $H_2$, leaving the displayed $2\times 2$ block uncoupled from them, and place every appended eigenvalue outside the open interval $(\lambda_2(A),\lambda_1(A))$.
Require also that each appended eigenvalue be at least $\eta$ away from every eigenvalue of the opposite diagonal block, so that the gap remains $\eta$.
The singular values of $\what E$ are then $\epsilon$, $\epsilon$, and zeros.
In the ordered spectra, the appended eigenvalues occupy the same positions for $A$ and $\wtd A$ and contribute zero to $\Delta\blambda$; the original pair contributes the two absolute changes $\phi(\eta,\epsilon)\epsilon$.
Thus $|\Delta\blambda|$ and $\phi(\eta,\epsilon)\bsigma(\what E)$ have the same decreasing rearrangement, and equality holds in \eqref{eq:main} for every unitarily invariant norm.
\end{example}

\section{Application to singular values}\label{sec:svd}
We apply \Cref{thm:main-UI} to rectangular matrices through the
Jordan--Wielandt dilation. As in \cite[Section~3]{lili:2005}, for
$X\in\bbC^{p\times q}$ we use the zero-padded singular-value vector
\[
 \bsigma(X)=[\sigma_1(X),\ldots,\sigma_{\max\{p,q\}}(X)]^{\T}.
\]
Its first $\min\{p,q\}$ entries are the usual singular values; the
remaining $|p-q|$ entries are zero. Equivalently, these are the
nonnegative square roots of the eigenvalues of the larger of
$XX^*$ and $X^*X$. We write $\sv(X)$ for the multiset of these
padded singular values. Including the zeros is essential when defining
the gap below for rectangular diagonal blocks.
 
For the norm comparison below, fix an integer
$D\ge\max\{m+n,k+\ell\}$ and one symmetric gauge function $\Phi$ on $\bbR^D$. Define the norm of each matrix of size at most $D\times D$ by zero padding into $\bbC^{D\times D}$ and applying the unitarily invariant norm induced by the associated $\Phi$. Both sides therefore use the same gauge function, and the comparison is unchanged by taking adjoints or by choosing different zero paddings within this fixed ambient space. The usual Schatten and Ky Fan norms use this convention.

\begin{theorem}\label{thm:main-sv}
Let $m,n,k,\ell$ be positive integers, and consider
\begin{equation}\label{eq:BtB-svd}
B = \kbordermatrix{  & \sss k & \sss \ell \cr
                  \sss m & G_1 & E_1  \cr
                  \sss n & E_2  & G_2 \cr }, \quad
\wtd B = \kbordermatrix{  & \sss k & \sss \ell \cr
                  \sss m & G_1 & 0  \cr
                  \sss n & 0  & G_2 \cr }
 \in\bbC^{(m+n)\times(k+\ell)}.
\end{equation}
Set $\Delta\bsigma=\bsigma(B)-\bsigma(\wtd B)$,
$\epsilon_2=\max\{\|E_1\|_2,\|E_2\|_2\}$, and
\[
 \eta=\min_{\mu_1\in\sv(G_1),\,\mu_2\in\sv(G_2)}
           |\mu_1-\mu_2|.
\]
Then
\begin{equation}\label{eq:sv-majorization}
 |\Delta\bsigma|\prec_w
       \phi(\eta,\epsilon_2)\bsigma(B-\wtd B).
\end{equation}
Consequently, for
every unitarily invariant norm with the convention above,
\begin{equation}\label{eq:main-sv}
 \|\diag(\Delta\bsigma)\|_{\UI}
 \le\phi(\eta,\epsilon_2)
       \left\|\begin{bmatrix}0&E_1\\E_2&0\end{bmatrix}\right\|_{\UI}.
\end{equation}
In particular,
\[
 \|\Delta\bsigma\|_2
 \le\phi(\eta,\epsilon_2)
       \bigl(\|E_1\|_{\F}^2+\|E_2\|_{\F}^2\bigr)^{1/2}.
\]
\end{theorem}

\begin{proof}
For $X\in\bbC^{p\times q}$, write
\[
 \scrJ(X)=\begin{bmatrix}0&X\\X^*&0\end{bmatrix},
 \quad d=\min\{p,q\}.
\]
By the Jordan--Wielandt theorem
\cite[Theorem~I.4.2]{stsu:1990}, the ordered eigenvalues of
$\scrJ(X)$ are
\[
 \sigma_1(X),\ldots,\sigma_d(X),\quad
 \underbrace{0,\ldots,0}_{|p-q|\text{ entries}},\quad
 -\sigma_d(X),\ldots,-\sigma_1(X).
\]
Apply this to $B$ and $\wtd B$, with $p=m+n$ and $q=k+\ell$.

{Permuting the coordinates into groups of sizes $m,k,n,\ell$ puts
$\scrJ(B)$ into the block form \eqref{eq:A}. In this proof, use the local notation}
\[
 H_i=\begin{bmatrix}0&G_i\\G_i^*&0\end{bmatrix},
 \quad
 E=\begin{bmatrix}0&E_2\\E_1^*&0\end{bmatrix},
\]
and puts $\scrJ(\wtd B)$ into the corresponding block-diagonal
form. Here $H_1$ and $H_2$ have sizes $m+k$ and $n+\ell$,
respectively, and $E\in\bbC^{(n+\ell)\times(m+k)}$.
The nonzero singular values of both $E$ and $B-\wtd B$ are those of
$E_1$ together with those of $E_2$. In particular,
$\|E\|_2=\epsilon_2$.

Every absolute eigenvalue of $H_i$ is a padded singular value of
$G_i$, and conversely. Every padded singular value of $G_i$ also
occurs as a nonnegative eigenvalue of $H_i$.
For any $\alpha\in\eig(H_1)$ and $\beta\in\eig(H_2)$,
\[
 |\alpha-\beta|\ge\bigl||\alpha|-|\beta|\bigr|\ge\eta.
\]
Conversely, nonnegative eigenvalues attain the minimum defining
$\eta$. Therefore
\[
 \gap(\eig(H_1),\eig(H_2))=\eta.
\]

Let $\Delta\blambda=\blambda(\scrJ(B))
-\blambda(\scrJ(\wtd B))$ and
\[
 \what E=\begin{bmatrix}0&E^*\\ E&0\end{bmatrix}.
\]
The ordered eigenvalue description shows that $|\Delta\blambda|$
contains each of $|\sigma_i(B)-\sigma_i(\wtd B)|$, $1\le i\le d$, twice,
together with zeros. The entries $\sigma_i(B)-\sigma_i(\wtd B)$ for $i>d$ are
zero. Thus, after extending the vectors by zeros, the decreasing
rearrangement of $|\Delta\blambda|$ consists of each entry of
$|\Delta\bsigma|^\downarrow$ repeated twice. By \Cref{lem:sV},
the singular values of $\what E$ likewise consist of each singular value
of $B-\wtd B$ repeated twice, after zero padding.

\Cref{thm:main-UI} gives
$|\Delta\blambda|\prec_w\phi(\eta,\epsilon_2)\bsigma(\what E)$.
This relation remains valid after appending zeros, so comparison of
the partial sums of even order yields, for every integer $\nu\ge1$,
\begin{align*}
 \sum_{i=1}^{\nu}|\Delta\bsigma|_i^\downarrow
 &=\frac12\sum_{j=1}^{2\nu}|\Delta\blambda|_j^\downarrow\\
 &\le\frac{\phi(\eta,\epsilon_2)}2
                  \sum_{j=1}^{2\nu}\sigma_j(\what E)
 =\phi(\eta,\epsilon_2)
                  \sum_{i=1}^{\nu}\sigma_i(B-\wtd B).
\end{align*}
This proves \eqref{eq:sv-majorization}, and Fan's dominance theorem
gives \eqref{eq:main-sv}. The Frobenius case follows by taking the
Euclidean symmetric gauge and using
$\|B-\wtd B\|_{\F}^2=\|E_1\|_{\F}^2+\|E_2\|_{\F}^2$.
\end{proof}

The classical Mirsky singular-value perturbation theorem \cite{mirs:1960}
(see also \cite[p.~204]{stsu:1990} and \cite{li:2014HLA}) gives
\begin{equation}\label{eq:Mirsky-svd}
 \|\diag(\Delta\bsigma)\|_{\UI}
 \le\left\|\begin{bmatrix}0&E_1\\E_2&0\end{bmatrix}\right\|_{\UI}.
\end{equation}
Since $\phi(\eta,\epsilon_2)\in[0,1]$, \eqref{eq:main-sv} is no
weaker than \eqref{eq:Mirsky-svd}. For $\eta>0$ and
$\epsilon_2/\eta$ small, the factor $\phi(\eta,\epsilon_2)$ is asymptotic to
$\epsilon_2/\eta$, making our new inequality \eqref{eq:main-sv} much sharper than
the classical \eqref{eq:Mirsky-svd}.

\section{Conclusion}
\label{sec:concl}
{For the partitioned Hermitian matrices $A$ and $\wtd A$ in \cref{sec:intro}, we have proved}
\begin{equation}\label{eq:lili-hope}
 \big\|\diag\big(\blambda(A)-\blambda(\wtd A)\big)\big\|_{\UI}
 \le\phi(\eta,\epsilon_2)\|A-\wtd A\|_{\UI}
\end{equation}
for every unitarily invariant norm. Li and Li~\cite{lili:2005} proved its spectral-norm case. Since $\phi(\eta,\epsilon_2)\le1$, the new bound improves the classical Lidskii--Mirsky--Wielandt bound \eqref{eq:LMW} and can be much smaller when $0<\epsilon_2/\eta\ll1$. The factor is sharp, as shown by \Cref{ex:sharp}.


 
The Jordan--Wielandt dilation transfers the result to singular values of rectangular block matrices, with the same factor and under the common zero-embedding convention in \cref{sec:svd}. Li and Li~\cite{lili:2005} also proved the spectral-norm case of this singular-value bound. The resulting inequality improves Mirsky's classical singular-value perturbation bound 
\eqref{eq:Mirsky-svd} in this block setting.
 
\section*{AI statement}
The authors used assistance from AI tools
to develop ideas presented in this paper. The
authors assume responsibility for all content.


\clearpage
{\small 
\begin{thebibliography}{10}

\bibitem{bhat:1997}
R.~Bhatia.
\newblock {\em Matrix Analysis}.
\newblock Graduate Texts in Mathematics, vol. 169. Springer, New York, 1997.

\bibitem{daka:1970}
C.~Davis and W.~Kahan.
\newblock The rotation of eigenvectors by a perturbation.~{III}.
\newblock {\em SIAM J. Numer. Anal.}, 7:1--46, 1970.

\bibitem{demm:1997}
J.~Demmel.
\newblock {\em Applied Numerical Linear Algebra}.
\newblock SIAM, Philadelphia, PA, 1997.

\bibitem{govl:2013}
G.~H. Golub and C.~F. {Van Loan}.
\newblock {\em Matrix Computations}.
\newblock Johns Hopkins University Press, Baltimore, Maryland, 4th edition,
  2013.

\bibitem{grlo:2020}
A.~Greenbaum, R.-C. Li, and M.~L. Overton.
\newblock First-order perturbation theory for eigenvalues and eigenvectors.
\newblock {\em SIAM Rev.}, 62(2):463--482, 2020.

\bibitem{heil:2019}
C.~Heil.
\newblock Absolute continuity and the fundamental theorem of calculus.
\newblock In {\em Introduction to Real Analysis}, volume 280 of {\em Graduate
  Texts in Mathematics}, chapter~6, pages 219--252. Springer, Cham, 2019.

\bibitem{howi:1953}
A.~J. Hoffman and H.~W. Wielandt.
\newblock The variation of the spectrum of a normal matrix.
\newblock {\em Duke Math. J.}, 20:37--39, 1953.

\bibitem{hojo:1991}
R.~A. Horn and C.~R. Johnson.
\newblock {\em Topics in Matrix Analysis}.
\newblock Cambridge University Press, Cambridge, 1991.

\bibitem{kato:1995}
T.~Kato.
\newblock {\em Perturbation Theory for Linear Operators}.
\newblock Springer-Verlag, Berlin, 1995.
\newblock Reprint of the 1980 edition.

\bibitem{lili:2005}
C.-K. Li and R.-C. Li.
\newblock A note on eigenvalues of perturbed {Hermitian} matrices.
\newblock {\em Linear Algebra Appl.}, 395:183--190, 2005.
\newblock There is a corrected arXiv version (2025).

\bibitem{lili:2025arXiv}
C.-K. Li and R.-C. Li.
\newblock A note on eigenvalues of perturbed {Hermitian} matrices, August 2025.
\newblock Corrected version of the 2005 article,
  \href{https://arxiv.org/abs/2508.08203}{arXiv:2508.08203}.

\bibitem{li:2014HLA}
R.-C. Li.
\newblock Matrix perturbation theory.
\newblock In L.~Hogben, R.~Brualdi, and G.~W. Stewart, editors, {\em Handbook
  of Linear Algebra}, page Chapter 21. CRC Press, Boca Raton, FL, 2nd edition,
  2014.

\bibitem{math:1998}
R.~Mathias.
\newblock Quadratic residual bounds for the {Hermitian} eigenvalue problem.
\newblock {\em SIAM J. Matrix Anal. Appl.}, 19:541--550, 1998.

\bibitem{mirs:1960}
L.~Mirsky.
\newblock Symmetric gauge functions and unitarily invariant norms.
\newblock {\em Quart. J. Math.}, 11:50--59, 1960.

\bibitem{parl:1998}
B.~N. Parlett.
\newblock {\em The Symmetric Eigenvalue Problem}.
\newblock SIAM, Philadelphia, 1998.
\newblock This SIAM edition is an unabridged, corrected reproduction of the
  work first published by Prentice-Hall, Inc., Englewood Cliffs, New Jersey,
  1980.

\bibitem{part:1978}
K.~R. Parthasarathy.
\newblock Eigenvalues of matrix-valued analytic maps.
\newblock {\em J. Austral. Math. Soc. Ser. A}, 26:179--197, 1978.

\bibitem{stsu:1990}
G.~W. Stewart and Ji-Guang Sun.
\newblock {\em Matrix Perturbation Theory}.
\newblock Academic Press, Boston, 1990.

\bibitem{neum:1937}
J.~{von Neumann}.
\newblock Some matrix-inequalities and metrization of matrix-space.
\newblock {\em Tomsk. Univ. Rev.}, 1:286--300, 1937.

\end{thebibliography}
}

\end{document}